\documentclass{article}

\usepackage[utf8]{inputenc}

\def\showauthornotes{0}
\def\showkeys{0}
\def\showdraftbox{0}

\usepackage{xspace,enumerate}
\usepackage{amsmath,amssymb}
\usepackage{amsthm}
\usepackage[toc,page]{appendix}
\usepackage{thmtools}
\usepackage{thm-restate}
\usepackage{color,graphicx}
\usepackage{boxedminipage}
\usepackage{makecell}
\usepackage{tabularx}

\ifnum\showkeys=1
\usepackage[color]{showkeys}
\fi

\definecolor{darkred}{rgb}{0.5,0,0}
\definecolor{darkgreen}{rgb}{0,0.35,0}
\definecolor{darkblue}{rgb}{0,0,0.55}

\usepackage[dvipsnames]{xcolor}

\usepackage[pdfstartview=FitH,pdfpagemode=None,colorlinks,linkcolor=NavyBlue,filecolor=blue,citecolor=OliveGreen,urlcolor=NavyBlue,pagebackref]{hyperref}

\usepackage[capitalise,nameinlink]{cleveref}
\usepackage[T1]{fontenc}
\usepackage{mathtools,dsfont,bbm}
\usepackage{mathpazo}
\usepackage{microtype}
\usepackage[top=1in, bottom=1in, left=1.25in, right=1.25in]{geometry}

\ifnum\showauthornotes=1
\newcommand{\Authornote}[3]{{\sf\small\color{#3}{[#1: #2]}}}
\newcommand{\Authorcomment}[2]{{\sf \small\color{gray}{[#1: #2]}}}
\newcommand{\Authorfnote}[2]{\footnote{\color{red}{#1: #2}}}
\else
\newcommand{\Authornote}[3]{}
\newcommand{\Authorcomment}[2]{}
\newcommand{\Authorfnote}[2]{}
\fi

\ifnum\showdraftbox=1
\newcommand{\draftbox}{\begin{center}
  \fbox{%
    \begin{minipage}{2in}%
      \begin{center}%
        \begin{Large}%
          \textsc{Working Draft}%
        \end{Large}\\
        Please do not distribute%
      \end{center}%
    \end{minipage}%
  }%
\end{center}
\vspace{0.2cm}}
\else
\newcommand{\draftbox}{}
\fi

\newtheorem{theorem}{Theorem}[section]
\newtheorem{conjecture}[theorem]{Conjecture}

\newtheorem{lemma}[theorem]{Lemma}

\newtheorem{corollary}[theorem]{Corollary}
\newtheorem{claim}[theorem]{Claim}

\newtheorem{algo}[theorem]{Algorithm}

\def\FullBox{\hbox{\vrule width 6pt height 6pt depth 0pt}}

\def\qed{\ifmmode\qquad\FullBox\else{\unskip\nobreak\hfil
\penalty50\hskip1em\null\nobreak\hfil\FullBox
\parfillskip=0pt\finalhyphendemerits=0\endgraf}\fi}

\def\qedsketch{\ifmmode\Box\else{\unskip\nobreak\hfil
\penalty50\hskip1em\null\nobreak\hfil$\Box$
\parfillskip=0pt\finalhyphendemerits=0\endgraf}\fi}

\newenvironment{proofof}[1]{\begin{trivlist} \item {\bf Proof of
#1:~~}}
  {\qed\end{trivlist}}

\def\eps{\varepsilon}
\def\epsilon{\varepsilon}

\def\eps{\epsilon}

\def\phi{\varphi}
\def\cal{\mathcal}

\newcommand{\given}{\;\ifnum\currentgrouptype=16 \middle\fi \vert\;}

\newcommand{\ie}{i.e.,\xspace}
\newcommand{\eg}{e.g.,\xspace}
\newcommand{\etal}{et al.\xspace}

\newcommand{\mper}{\,.}
\newcommand{\mcom}{\,,}

\newcommand{\R}{{\mathbb R}}

\newcommand{\indicator}[1]{\mathds{1}_{\{#1\}}}

\newcommand{\gaussian}[2]{{\cal N(#1, #2)}}

\usepackage{nicefrac}

\let\nfrac=\nicefrac

\newcommand{\abs}[1]{\ensuremath{\left\lvert #1 \right\rvert}}

\newcommand{\norm}[1]{\ensuremath{\left\lVert #1 \right\rVert}}

\newcommand{\ip}[2] {\ensuremath{\left\langle #1 , #2 \right\rangle}}

\newcommand{\Esymb}{\mathbb{E}}
\newcommand{\Psymb}{\mathbb{P}}

\newcommand{\Varsymb}{\mathrm{Var}}
\DeclareMathOperator*{\ExpOp}{\Esymb}

\makeatletter
\def\Pr#1{%
    \ProbabilityRender{\Psymb}{#1}%
}

\def\Ex#1{%
    \ProbabilityRender{\Esymb}{#1}%
}

\def\condPE#1#2{%
	\@ifnextchar\bgroup
	{\ConditionalProbabilityRender{\widetilde{\Esymb}}{#1}{#2}}
	{\ProbabilityRender{\widetilde{\Esymb}}{#1 \given #2}}
}

\def\Var#1{%
    \ProbabilityRender{\Varsymb}{#1}%
}

\def\ConditionalProbabilityRender#1#2#3#4{
	\renderwithdist{#1}{#2}{#3 \given #4}	
}

\def\ProbabilityRender#1#2{
  \@ifnextchar\bgroup%
  {\renderwithdist{#1}{#2}}
   {\singlervrender{#1}{#2}}
}
\def\singlervrender#1#2{%
   \ensuremath{\mathchoice
       {{#1}\left[ #2 \right]}
       {{#1}[ #2 ]}
       {{#1}[ #2 ]}
       {{#1}[ #2 ]}
   }
}
\def\renderwithdist#1#2#3{%
   \@ifnextchar\bgroup
   {\superfancyrender{#1}{#2}{#3}}
   {\ensuremath{\mathchoice
      {\underset{#2}{#1}\left[ #3 \right]}
      {{#1}_{#2}[ #3 ]}
      {{#1}_{#2}[ #3 ]}
      {{#1}_{#2}[ #3 ]}
     }
   }
}
\def\superfancyrender#1#2#3#4#5{
   \ensuremath{\mathchoice
      {\underset{#1}{{#1}}\left#4 #3 \right#5}
      {{#1}_{#2}#4 #3 #5}
      {{#1}_{#2}#4 #3 #5}
      {{#1}_{#2}#4 #3 #5}
   }
}
\makeatother

\def\expop{\ExpOp}

\newfont{\inhead}{eufm10 scaled\magstep1}

\DeclareMathOperator\supp{Supp}

\newcommand{\inparen}[1]{\left(#1\right)}             
\newcommand{\inbraces}[1]{\left\{#1\right\}}           

\usepackage{url}	

\newcommand*\diff{\mathop{}\!\mathrm{d}}

\newcommand{\trans}[1]{{#1}^{\mathsf{T}}}

\title{Ellipsoid fitting up to constant \\ via empirical covariance estimation}
\author{
Madhur Tulsiani\thanks{{\tt TTIC}. {\tt madhurt@ttic.edu}. Research partly supported by NSF grant CCF-1816372.} 
\and
June Wu \thanks{{\tt University of Chicago}. {\tt jqw@uchicago.edu}. } 
}

\allowdisplaybreaks

\begin{document}

\date{}

\maketitle
\draftbox

\thispagestyle{empty}
\begin{abstract}
The ellipsoid fitting conjecture of Saunderson, Chandrasekaran, Parrilo and Willsky considers the maximum number $n$ random Gaussian points in $\R^d$, such that with high probability, there exists an origin-symmetric ellipsoid passing through all the points. They conjectured a threshold of $n = (1-o_d(1)) \cdot d^2/4$, while until recently, known lower bounds on the maximum possible $n$ were of the form $d^2/(\log d)^{O(1)}$.

We give a simple proof based on concentration of sample covariance matrices, that with probability $1 - o_d(1)$, it is possible to fit an ellipsoid through $d^2/C$ random Gaussian points.
Similar results were also obtained in two recent independent works by Hsieh, Kothari, Potechin and Xu [arXiv, July 2023] and by Bandeira, Maillard, Mendelson, and Paquette [arXiv, July 2023].

\end{abstract}
\newpage


\pagenumbering{arabic}
\setcounter{page}{1}


\section{Introduction}
In this note, we consider the problem of fitting an ellipsoid to random Gaussian points in $d$ dimensions. 
An origin-symmetric ellipsoid in $\R^d$ is specified by the equation $\ip{x}{Qx} = 1$, for a positive semidefinite matrix $Q \succeq 0$. The following problem was considered by Saunderson, Chandrasekaran, Parrilo, and Willsky \cite{SCPW12, SPW13} in the context of recovering a subspace from noisy measurements, using semidefinite programs:
Given $z_1, \ldots, z_n \in \R^d$, does there exist $Q \succeq 0$ such that $\ip{z_i}{Qz_i} = 1 ~~\forall i \in [n]$?

They proved that for a $d$-dimensional subspace, fitting the projections of $n$ canonical basis vectors $V$ to an ellipsoid, is equivalent to recovering $V^{\perp}$ using a canonical semidefinite program. Thus, one is interested in determining how large can $n$ be taken relative to $d$, to ensure small co-dimension for the subspace $V^{\perp}$ (see \cite{Saunderson11} for a detailed discussion).

Saunderson \etal considered both average-case and worst-case versions of the above problem, and here we discuss average-case version where the points $z_1, \ldots, z_n$ can be taken to be independent samples from a Gaussian distribution. Since the problem is invariant under linear transformations (as $Q$ can be modified to $LQ\trans{L}$ for any linear transformation $L$), it is convenient to consider $z_1, \ldots, z_n \sim \gaussian{0}{\frac{1}{d} \cdot I_d}$, so that in expectation, the points lie on the unit sphere $S^{d-1}$. One then considers the problem of determining the largest $n$ as a function of $d$, such that with probability close to 1 (say $1 - o_d(1)$), there exists an ellipsoid passing through the random points $z_1, \ldots, z_n$. It is easy to see that we must have $n < \binom{d+1}{2}$ via dimension arguments. Saunderson \etal~\cite{SCPW12, Saunderson11} made the following conjecture based on experimental results.
\begin{conjecture}[\cite{SCPW12, Saunderson11}]
For arbitrarily small $\eps > 0$, if $n < (\nfrac14-\eps) \cdot d^2$, then with high probability over $z_1, \ldots, z_n \sim \gaussian{0}{\frac{1}{d} \cdot I_d}$, there exists an ellipsoid passing through all the points.
\end{conjecture}

In addition to being an intriguing conjecture in its own right, this problem has also attracted recent attention because of its connections to techniques used in the study of lower bounds for semidefinite programs. 
In such lower bounds, the key technical challenge is often to come up with a candidate matrix, \emph{which may depend on random data}, and then to prove the positivity of such a matrix (with high probability). 
The task of finding the matrix $Q \succeq 0$ also requires exhibiting a solution to a semidefinite program, and in fact some of the known results for the ellipsoid fitting problem have a significant overlap with semidefinite programming lower bounds, in terms of the techniques.
\vspace{-10 pt}
\paragraph{Known results.} Saunderson \etal \cite{SPW13} proved the ellipsoid fitting property (with high probability) for $n$ random points when $n \leq d^{6/5 - \eta}$ for arbitrarily small $\eta>0$. 
This bound was improved to $n \leq d^{3/2 - \eta}$ by a result of Ghosh \etal~\cite{GJJPR20} in the context of lower bounds for the Sum-of-Squares semidefinite programming hierarchy, for the closely related ``Planted Affine Planes" problem. 
These bounds were further improved to $n \leq d^2/(\log d)^{O(1)}$ by two independent works of Potechin \etal ~\cite{PTVW22}, and of Kane and Diakonikolas \cite{DK22}. The work of Potechin \etal relied on techniques used in proofs of Sum-of-Squares lower bounds, such as decompositions of random matrices into so called ``graph matrices" (and a careful analysis of the coefficients in such decompositions, as well as trade-off between various terms~\cite{AMP21, PR20}). On the other hand, Kane and Diakonikolas gave a more direct proof based on the spectral analysis of a single random matrix arising from the random data.
\vspace{-10 pt}
\paragraph{Recent and independent work.} Very recently, two independent works by Hsieh, Kothari, Potechin and Xu \cite{HKPX23}, and by Bandeira, Maillard, Mendelson, and Paquette \etal~\cite{BMMP23} improved the bound on the number of points to $n \leq d^2/C$ for some (possibly large) constant $C > 0$, in both cases by obtaining a sharper analysis of the construction in \cite{DK22}. 
The result of Hsieh \etal relies on a more careful application of the graph matrix decompositions arising in \cite{PTVW22}, while the work of Bandeira \etal uses a different approach based on concentration bounds for Gram matrices of independently sampled vectors from distributions with good tail behavior. 

In this note, we present independent results, also proving that with high probability, there exists an ellipsoid passing through $n$ random points, for $n \leq d^2/C$ (see \cref{thm:main} for a formal statement of our result). 
The approach in this work also gives a sharper analysis of the construction of Kane and Diakonikolas, obtaining spectral norm bounds for random matrices arising in their construction, by using results on the concentration of empirical covariance matrices. 
Similar ideas were also used independently in the work of Bandeira \etal~\cite{BMMP23}, although the approaches differ in the specifics of the concentration results used to obtain the relevant norm bounds.

\subsection{Preliminaries and notation}
The Euclidean norm $\|\cdot\|_2$ and scalar product $\langle \cdot , \cdot \rangle$ in $\mathbb{R}^d$ are defined here using the counting measure on the coordinates. For a matrix $A \in \R^{m \times n}$, $\norm{A}$ denotes its spectral norm, and $\norm{A}_F$ denotes the Frobenius norm.
For a real symmetric matrix $B$, $diag(B)$ denotes the diagonal matrix obtained from $B$ replacing all non-diagonal entries by $0$. 
We denote various positive absolute constants by $C, C', C_0, C_1, \ldots$ when the constants are greater than 1, and  $c, c', c_0, c_1, \ldots$ when they are smaller than 1. The values of the constants may change from one instance to another.
\vspace{-10 pt}
\paragraph{Orlicz norms, sub-gaussian and sub-exponential random variables.}
Let $X$ be a real-valued random variable with mean zero. We define the Orlicz $\alpha$-norm $\psi_\alpha$ for $\alpha \geq 1$ as
$$
\|X\|_{\psi_\alpha} = \inf \{t > 0: \mathbb{E} \exp(|X|^\alpha / t^\alpha) \leq 2 \}.
$$
Let $Z$ be a random vector in $\mathbb{R}^n$. $\psi_\alpha$ norm of a random vector is characterized by its one-dimensional projections,
$$
\|Z\|_{\psi_\alpha} = \sup_{x\in S^{n-1}} \|\langle Z, x\rangle \|_{\psi_\alpha}.
$$
A random variable $X$ with bounded $\norm{X}_{\psi_1}$ is known as a sub-exponential random variable, while one with bounded $\norm{X}_{\psi_2}$ is known as a sub-gaussian random variable. 
The above definitions are also known to be equivalent to decay of the the tail probabilities $\Pr{\abs{X} \geq t}$ as $\exp\inparen{- t/a_1}$ for sub-exponential random variables, and $\exp\inparen{- t^2/a_2^2}$ for sub-gaussian random variables, where $a_1$ and $a_2$ are within constant factors of the $\psi_1$ and $\psi_2$ norms respectively (see \eg \cite{Ver18}). 

\section{Fitting an ellipsoid to $d^2/C$ random points}
We will prove that with high probability over the choice of $n$ independent Gaussian vectors in $\R^d$, there exists an origin-symmetric ellipsoid passing through all the points, when $n \leq d^2/C$ for some fixed constant $C$.
\begin{theorem} [Main result]
\label{thm:main}
There exists a universal constant $C > 4$ such that for $n \leq d^2/C$, and with probability $1 - o_d(1)$ over the choice of i.i.d. samples $G_1, \ldots, G_n$ from $\gaussian{0}{\frac{1}{d} \cdot I_d}$, there exists a positive semidefinite matrix $Q \succeq 0$ satisfying
$\ip{G_i}{Q G_i} ~=~ 1 ~~\forall i \in [n]$.
\end{theorem}
We use the identity-perturbation construction of Kane and Diakonikolas~\cite{DK22}, which can be derived by considering the dual of the semidefinite program for finding the matrix $Q$ as described above. Since the program has a constraint $\ip{G_i}{Q G_i} ~=~ 1$ for each $i \in [n]$, the construction is defined by dual variables $\delta_i \in \R$ for each $i \in [n]$.
In particular, let $G_i = \norm{G_i}_2 \cdot X_i$, where $X_i \in S^{d-1}$ are random unit vectors, independent of the norms $\norm{G_i}_2$. Kane and Diakonikolas consider the matrix
\[
Q ~=~ I_d + \sum_{i = 1}^n \delta_i \cdot X_i \trans{X_i} \mcom
\] 
and show that with high probability the (unique) choice of the variables $\delta_i$ satisfying $\ip{G_i}{Q G_i} = 1$ also satisfies $Q \succeq 0$. Considering the linear constraints $\ip{G_i}{Q G_i} = 1$ on the matrix $Q$ above yields the linear system $M \delta = \epsilon$, for the matrix $M$ and the vector $\epsilon$ defined as
\[
M_{ij} ~=~ \ip{X_i}{X_j}^2 ~=~ \ip{X_i^{\otimes 2}}{X_j^{\otimes 2}}\qquad \text{and} \qquad \epsilon_i ~=~ \frac{1}{\norm{G_i}_2^2} - 1 \mper
\]
The proof then has two key components. The first is to show that the matrix $M$ is invertible, and in fact the random matrix $M^{-1}$ has bounded spectral norm with high probability. The second is to obtain bounds on the random vector $\epsilon$ (independent of $M$) of deviations in norms of $G_i$, to bound the norm of the perturbation $\sum_i \delta_i \cdot X_i \trans{X_i}$.

It is in the first part of the proof that our approach differs significantly from that of Diakonikolas and Kane~\cite{DK22}. Their proof relied on using the trace power method (which is the most technically involved part of the proof) and yielded bounds for $n \leq d^2/(\log d)^{O(1)}$. On the other hand, we rely on the fact that the vectors $X_i^{\otimes 2}$ have a sub-exponential distribution, and use existing results on the concentration of empirical covariance matrices for random sub-exponential vectors. 
The second part of our proof is similar to that of Kane and Diakonikolas, but needs to develop improved tail estimates for the vector $\epsilon$, to avoid the loss of $(\log d)^{O(1)}$ factors.

\subsection{Spectral norm bound for the Gram matrix}
We will bound the norm of $M^{-1}$ for the Gram matrix $M$ defined as $M_{ij} = \ip{X_i^{\otimes 2}}{X_j^{\otimes 2}}$ by using the results of Adamczak \etal~\cite{ALPTJ10a, ALPTJ10b} on the concentration of empirical covariance matrices, for independent sub-exponential random vectors. 
For independent (mean-zero) random vectors $Y_1, \ldots, Y_n$, forming columns of a matrix $Y$, the empirical covariance matrix is given (up to scaling) by $Y\trans{Y}$. Since $Y\trans{Y}$ has the same non-zero eigenvalues as the the Gram matrix $\trans{Y}Y$, their proofs in fact proceed by understanding the spectrum of the Gram matrix, which implies the results needed for the matrix $M$ above.
We will need the following inequality.
\begin{theorem}[Hanson-Wright Inequality, see e.g. \cite{Ver18} Theorem 6.2.1] 
\label{hanson-wright}
Let $Z = (z_1, \dots, z_d) \in \mathbb{R}^d$ be a random vector with independent mean-zero sub-gaussian coordinates. Let $A$ be a $d \times d$ matrix. Then, for every $t \geq 0$, we have
$$
\Pr{\abs{\trans{Z}AZ - \Ex{\trans{Z}AZ}} \geq t} ~\leq~ 2 \exp \left ( -c \min \left( \frac{t^2}{K^4\|A\|^2_{F}}, \frac{t}{K^2\|A\|} \right) \right ),
$$
where $K = \max_i \|z_i\|_{\psi_2}$.
\end{theorem}
We now consider the centered versions of the vectors $X^{\otimes 2}$, defined as $Y = X^{\otimes 2} - \expop X^{\otimes 2}$, and show for any unit vector $u \in \R^{d^2}$, the random variable $\ip{u}{Y}$ has sub-exponential tail. Since a vector $u \in \R^{d^2}$ can be thought of as a matrix $A \in \R^{d \times d}$, we can obtain the required tail behavior using \cref{hanson-wright}.
\begin{lemma} [Sub-exponential behavior]
\label{subexponential}
Let $X = (x_1, \dots, x_d)$ be a random vector uniformly distributed on the unit sphere: $X \sim Unif(S^{d-1})$. Then $Y = X^{\otimes 2} - \mathbb{E} X^{\otimes 2}$ is a sub-exponential random vector satisfying $\norm{X^{\otimes 2} - \mathbb{E} X^{\otimes 2}}_{\psi_1} \leq C/d.$
\end{lemma}
\begin{proof}
Consider a standard normal random vector $G \sim \gaussian{0}{I_d}$. The direction $X = G /\norm{G}_2$ is uniformly distribution on $S^{d-1}$. We need to show that $\langle X^{\otimes 2} - \mathbb{E} X^{\otimes 2}, u \rangle$ is sub-exponential for all unit vectors $u \in \mathbb{R}^{d^2}$. Reshape $u$ into a $d \times d$ matrix $A$ in which $a_{ij} = \langle u, e_i \otimes e_j \rangle$. Hence
\[
\langle X^{\otimes 2} - \mathbb{E} X^{\otimes 2}, u \rangle 
~=~ \sum_{i,j = 1}^d (x_ix_j - \mathbb{E} x_ix_j) \langle u, e_i \otimes e_j \rangle  
~=~ \trans{X}AX - \mathbb{E}\trans{X}AX.
\]
We want to bound the tail probability,
\[
\Pr{|\langle X^{\otimes 2} - \mathbb{E} X^{\otimes 2}, u \rangle| \geq \frac{t}{d}} 
~=~ \Pr{| \trans{X}AX - \mathbb{E}\trans{X}AX| \geq \frac{t}{d}} ~=~ 
\Pr{\left \vert \frac{\trans{G}AG}{\|G\|_2^2} - \mathbb{E}\frac{\trans{G}AG}{\|G\|_2^2} \right \vert \geq \frac{t}{d}}
\]
Notice that $\|G\|_2$ is well concentrated around $\sqrt{d}$. Thus the event $E:= \{ \|G\|_2 \geq \sqrt{d}/2\}$ is very likely and its complement $\mathbb{P}(E^c)\leq 2 \exp(-c_1 \cdot d)$. Then 
\begin{align*}
\Pr{\abs{\frac{\trans{G}AG}{\|G\|_2^2} - \mathbb{E}\frac{\trans{G}AG}{\|G\|_2^2} } \geq \frac{t}{d}} 
&~\leq~ \Pr{ \inbraces{ \abs{\frac{\trans{G}AG}{\|G\|_2^2} - \mathbb{E}\frac{\trans{G}AG}{\|G\|_2^2} } \geq \frac{t}{d} } \wedge E} + \mathbb{P} (E^c) \\
&~\leq~ \Pr{\abs{\trans{G}AG - \mathbb{E}\trans{G}AG }  \geq \frac{t}{4} } + 2 \cdot \exp(-c_1 \cdot d) \\
&~\leq~ 2 \exp \left (-c_0 \cdot \min ( t^2, t )\right) + 2 \exp(-c_1 \cdot d)\mcom
\end{align*}
where the last inequality is due to \cref{hanson-wright} with $\|A\| \leq \|A\|_F = 1$ and $K = \max_i \|G_i\|_{\psi_2} = 1$. 

For the first term in the last inequality, take $c_2$ such that $2\exp(-c_2) \geq 1$ and let $c' = \min(c_0,c_2)$. The first term is bounded by $2\exp(-c'\cdot t)$. 
For the second term in the last inequality, notice that $|\langle X^{\otimes 2}, u \rangle| \leq \|X^{\otimes 2}\|_2\|u\|_2 = 1$. This implies that $t \leq 2d$, otherwise the tail probability $\mathbb{P} \left(|\langle X^{\otimes 2} - \mathbb{E} X^{\otimes 2}, u \rangle| > 2 \right) = 0$. If $t \leq 2d$, then $2 \exp(-c_1 \cdot d) \leq 2 \exp(-(c_1/2) \cdot t)$. Now let $c = \min(c_1/2, c')$ and substitute $\theta = t/d \geq 0$. It follows that
$$
\mathbb{P} \left(|\langle X^{\otimes 2} - \mathbb{E} X^{\otimes 2}, u \rangle| \geq \theta \right) \leq 4 \exp(-c \cdot \theta d).
$$
It is a standard argument (see e.g. \cite{Ver18} Proposition 2.7.1) that if the tail probability is bounded by $2\exp(-c\theta d)$ for all $\theta \geq 0$, then $\|X^{\otimes 2} - \mathbb{E} X^{\otimes 2} \|_{\psi_1} \leq C/d$.
\end{proof}
With the above lemma, we can now apply a result of Adamczak \etal~\cite{ALPTJ10a} which yields bounds on the singular values of a random matrix $Y$ with columns of $Y$ being independent sub-exponential random variables. We present a simplified version of Theorem 3.13 in \cite{ALPTJ10a} below.
%
\begin{theorem}[\cite{ALPTJ10a}, Simplified version of Theorem 3.13] 
\label{singular value}
Let $Y_1, \dots, Y_N$ be independent random vectors in $\mathbb{R}^D$ such that $\max_{1\leq i \leq N}\|Y_i\|_{\psi_1} \leq \psi$ and $\max_{1 \leq i \leq N} \|Y_i\|_2 \leq 1$. Let $Y$ be a random $D \times N$ matrix whose columns are $Y_i$'s. Then 
\[
\Pr{ \sigma_1(Y) \leq C\psi \cdot \sqrt{N} + 6 } ~\geq~ 1 - \exp(-c \sqrt{N}) \mper   
\]
\end{theorem}
For our application, we will actually need a bound on the singular values which decreases with $\psi$, and it is in fact easy to check that the additive term of 6 in their bound arises from diagonal entries of the matrix $\trans{Y}{Y}$, which can be $\Omega(1)$ if the columns are unit vectors. We will need the following version of their result, bounding the contribution of the off-diagonal terms.
\begin{lemma} [Off-diagonal estimate] 
\label{off-diagonal}
Let $Y_1, \dots, Y_N$ be independent random vectors in $\mathbb{R}^D$ such that $\max_{1\leq i \leq N}\|Y_i\|_{\psi_1} \leq \psi$ and $\max_{1 \leq i \leq N} \|Y_i\|_2 \leq 1$. Let $Y$ be a random $n \times N$ matrix whose columns are $Y_i$'s. Then 
\[
\Pr{ \norm{\trans{Y}Y - diag(\trans{Y}Y)} \leq  C \max \{ \psi^2 \cdot N, \psi \cdot \sqrt{N}\} } ~\geq~ 1- \exp( - c \cdot \sqrt{N}) \mper
\]
\end{lemma}
\begin{proof}
We will use the tail bound for $\sigma_1(Y)$ from \cref{singular value} to obtain a tail bound for $\|\trans{Y}Y - diag(\trans{Y}Y)\|$. 
In the proof of Theorem 3.13 \cite{ALPTJ10a}, $\norm{\trans{Y}Y - diag(\trans{Y}Y)}$ and $\norm{diag(\trans{Y}Y)}$ are estimated separately in terms of $\sigma_1(Y)$. They are combined at the last step to give the bound on $\sigma_1(Y)$ by the triangle inequality
 $\|\trans{Y}Y\| \leq \|\trans{Y}Y - diag(\trans{Y}Y)\| + \|diag(\trans{Y}Y)\|$ together with an approximation argument. 
To simplify the bound presented in the proof of Theorem 3.13 \cite{ALPTJ10a}, we set the parameters exactly the same as in \cref{singular value}:
\begin{equation}
\label{eq2}
\Pr{ \norm{\trans{Y}Y - diag(\trans{Y}Y)} \leq  C_0 \sigma_1(Y) \cdot \psi \sqrt{N} } ~\geq~ 1 - \exp(-c_0 \cdot \sqrt{N}).
\end{equation}
Let $E_1 := \{\|\trans{Y}Y - diag(\trans{Y}Y)\| \leq  C_0 \sigma_1(A) \cdot \psi \sqrt{N} \}$ and $E_2 := \{ \sigma_1(Y) \leq C\psi \cdot \sqrt{N} + 6  \}$. 
\[
\mathbb{P}(E_1 \wedge E_2) ~=~ \mathbb{P}(E_1) - \mathbb{P}(E_1 \wedge E_2^c) ~\geq~ \mathbb{P}(E_1) - \mathbb{P}(E_2^c)
\]
Substituting the probability bounds from \cref{singular value} and \cref{eq2}, we have
\[
\norm{\trans{Y}Y - diag(\trans{Y}Y) } \leq  C_0  \cdot \psi \sqrt{N} \cdot (C  \psi \sqrt{N} + 6) ~\leq~ C_1 \cdot \max \{ \psi^2 N, \psi \sqrt{N}\} 
\]
with high probability. 
\end{proof} 
We can now apply the bound from the above lemma to control the spectral norm of $M^{-1}$ for the Gram matrix $M$. We denote the event that $\norm{M^{-1}} \leq 3$ by $E_1$.
\begin{lemma} [Operator norm bound: \underline{Event $E_1$}]
\label{operator}
Let $M$ be the matrix defined as $M_{ij} = \ip{X_i^{\otimes 2}}{X_j^{\otimes 2}}$, for $X_1, \ldots, X_n \sim \text{Unif}(S^{d-1})$. 
Then, for $n = d^2/C$ for a sufficient large constant $C$, the matrix $M$ satisfies $M \succeq \frac13 \cdot I_n + \frac{1}{d} \cdot J_n$, and thus $\|M^{-1}\| \leq 3$, with probability at least $1 - \exp(-c \cdot d)$.  
\end{lemma}
\begin{proof}
Let $Y_i = X_i^{\otimes 2} - \expop X_i^{\otimes 2}$, where $\expop X_i^{\otimes 2} = \frac{1}{d}\sum_{r=1}^d e_r^{\otimes 2}$. Note that
\[
\trans{Y}{Y} ~=~ M + \frac{1}{d} \cdot J_n 
\qquad  \text{and} \qquad 
\expop M ~=~ \inparen{1 - \frac{1}{d}} \cdot I_n + \frac{1}{d} \cdot J_n \mcom
\]
where $J_n$ denotes the all-ones matrix in $\R^{n \times n}$. Since the diagonal entries of $M$ are 1, we also get that $diag(\trans{Y}{Y}) = \inparen{1 - \frac{1}{d}} \cdot I_n $, and thus $\trans{Y}{Y} - diag(\trans{Y}{Y}) = M - \expop M$.
%

By \cref{subexponential}, $\inbraces{Y_i}_{i \in [n]}$ are independent sub-exponential random vectors with $\|Y_i\|_{\psi_1} \leq C/d$. 
%
%
Thus \cref{off-diagonal} implies that
\[
\Pr{ \|\trans{Y}{Y} - diag(\trans{Y}{Y})\| \leq  C \max \left \{ \frac{n}{d^2}, \frac{\sqrt{n}}{d} \right \} } 
~\geq~ 
1- \exp( - c\sqrt{n}).
\]
Choose $n = d^2/C_1$ for large enough $C_1$ such that $C \sqrt{n} / d \leq 1/2$. We have that with probability at least $1-\exp(-c \cdot d)$,
\[
\|\trans{Y}Y - diag(\trans{Y}Y)\| ~=~ \| M - \mathbb{E}M\| ~\leq~  1/2 \mper
\]
This implies
\[
M ~\succeq~ \inparen{\frac12 - \frac{1}{d}} \cdot I_n + \frac{1}{d} \cdot J_n ~\succeq~ \frac13 \cdot I_n
\]
which gives $\norm{M^{-1}} \leq 3$.
%
%
%
\end{proof}

\subsection{Estimates for deviations in norms}
We next develop some bounds for the random variables $\eps_i = \frac{1}{\norm{G_i}_2^2} - 1$, where $G_i \sim \gaussian{0}{\nfrac{1}{d}\cdot I_d}$. 
Together with the bounds for $\norm{M^{-1}}$ developed above, these will help us analyze the perturbation depending on the vector $\delta = M^{-1} \epsilon$. We also note that the $M$ and $\epsilon$ are independent random variables.
\begin{lemma} [Sub-gaussian behavior of $\epsilon$]
\label{subgaussian}
The independent random variables $\epsilon_1, \ldots, \epsilon_n$ all satisfy
\begin{enumerate}
\item $\Ex{\eps_i} = \frac{2}{d-2}$ ~and~ $\Var{\eps_i} = O(\frac{1}{d})$.
\item $\Pr{\abs{\eps_i - \expop \eps_i} \geq t} ~\leq~ 2 \cdot \exp\inparen{-t^2 \cdot d/32}$ for all $t \leq 1/2$.
\end{enumerate}
\end{lemma}
\begin{proof}
We fix an $i \in [n]$ and write $\eps_i = \frac{d}{\norm{G}^2_2} - 1$ for $G \sim \gaussian{0}{I_d}$. The random variable $1/\norm{G}^2_2$ has the inverse-$\chi^2$ distribution with $d$ degrees of freedom, for which the moments are known to be $\Ex{(1/\norm{G}^2_2)^k} = 2^{-k} \cdot \frac{\Gamma((d/2) - k)}{\Gamma(d/2)}$ when $k \leq d/2$ ~\cite{wiki1, wiki2}. A direct computation yields
\[
\Ex{\eps_i} ~=~ \frac{2}{d-2} \qquad \text{and} \qquad \Var{\eps_i} ~=~ \frac{2 \cdot d^2}{(d-2)^2 \cdot (d-4)} ~=~ O(1/d) \mper
\]
Let $\epsilon_0 = \mathbb{E}\epsilon_i$ and $Z_i = \epsilon_i - \mathbb{E}\epsilon_i = \epsilon_i - \epsilon_0 $. We want to bound the tail probability of $Z_i$. We have
\[
\Pr{Z_i \geq t} 
~=~ \Pr{ \frac{d}{\|G\|^2_2} - 1 \geq t + \epsilon_0}
~=~ \Pr{ \|G\|^2_2 - d \leq -\frac{d(t + \epsilon_0)}{1+t+\epsilon_0} },
\]
and
\[
\Pr{Z_i \leq -t} 
~=~ \Pr{ \frac{d}{\|G\|^2_2} - 1 \leq -t + \epsilon_0 }
~=~ \Pr{ \|G\|^2_2 - d \geq \frac{d(t - \epsilon_0)}{1-t+\epsilon_0} }.
\]
The tail bound for $\chi^2$ distribution (see \cite{Ver18} Theorem 3.1.1) implies that
\[
\Pr{|Z_i| \geq t} ~\leq~ 2 \cdot \exp \left( -\frac{d}{8} \min \left \{ \frac{t - \epsilon_0}{1 - t + \epsilon_0}, \frac{t + \epsilon_0}{1 + t + \epsilon_0} \right\}^2  \right).
\]
Using the fact that 
$$
\min \left \{ \frac{t - \epsilon_0}{1 - t + \epsilon_0}, \frac{t + \epsilon_0}{1 + t + \epsilon_0} \right\}^2  ~\geq~ \left(\frac{t}{2} \right)^2, \hspace{10pt} \forall t \in  [2\epsilon_0, 1/2], 
$$
we can simplify the tail bound to
$$
\mathbb{P}\left(|Z_i| \geq t \right) ~\leq~ 2 \exp \left( -\frac{dt^2}{32} \right), \hspace{10pt} \forall t \leq \frac{1}{2},
$$
since the probability is trivial for $t \leq 2\epsilon_0$. 
%
\end{proof}
The above lemma also shows that with high probability $\norm{\eps}_\infty = O(\sqrt{(\log d)/d})$, which we denote as the event $E_2$.
Actually, we will use $E_2$ as a shorthand for $\bigwedge E_{2,i}$, where $E_{2,i}$ is defined by a bound on $|\eps_i|$ (so that the collection of random variables $\inbraces{\eps_i | E_{2,i}}_{i \in [n]}$ remain independent).
%
\begin{corollary}[$\ell_{\infty}$ bound on $\epsilon$: \underline{Event $E_2$}]
\label{bound1}
With probability $1 - d^{-\Omega(1)}$ we have $\norm{\eps}_{\infty} \leq C \sqrt{\log(d)/d}$. Also, conditioned on $E_2$, we have
$\norm{\eps_i - \expop \eps_i}_{\psi_2} \leq C/\sqrt{d}$.
\end{corollary}
\begin{proof}
Using the tail bound from \cref{subgaussian}, and $n \leq d^2$ we have that
\[
\Pr{\exists i \in [n]. ~~\abs{\eps_i} \geq C_0 \cdot \frac{\sqrt{\log d}}{\sqrt{d}}} ~\leq~ d^2 \cdot \exp\inparen{-C_1 \cdot \log d} ~\leq~ d^{-C} \mper
\]
Also, the tail bound for $Z_i = \eps_i - \expop \eps_i$, conditioned on $E_2$ is:
\[
\Pr{|Z_i| \geq t ~|~ E_2} ~=~ \frac{\Pr{|Z_i| \geq t \wedge E_2}}{\Pr{E_2}} ~\leq~ \frac{2\exp(-dt^2/32)}{1-o_d(1)} ~\leq~ 2\exp(-dt^2/64), \hspace{10pt} \forall t.
\]
The $\psi_2$ norm follows directly from the tail bound.
\end{proof}


\begin{lemma} [$\ell^\infty$ bound on $\delta$: \underline{Event $E_3$}]
\label{bound2}
With probability $1 - o_d(1)$, the vector $\delta = M^{-1} \eps$ satisfies $\norm{\delta}_{\infty} \leq C \cdot (\log d)/\sqrt{d}$.
%
\end{lemma}
\begin{proof}
We condition on events $E_1$ and $E_2$ both of which happen with probability $1 - o_d(1)$. 
We fix the matrix $A = M^{-1}$ satisfying $\norm{A} \leq 3$, and consider $\epsilon$ (which is independent of $M$) with independent entries $\epsilon_1, \ldots, \epsilon_n$ satisfying $\norm{\eps_i - \expop \eps_i}_{\psi_2} \leq C/\sqrt{d}$ by \cref{bound1}. Then $\delta_i = \ip{A_i}{\eps}$, where $A_i$ denotes the $i$-th row of $A$ with $\norm{A_i}_2 \leq \norm{A} \leq 3$. 
By Hoeffding's inequality for independent sub-gaussian variables
\begin{align*}
\Pr{\abs{\delta_i - \expop \delta_i} \geq t}
~=~ \Pr{\abs{\ip{A_i}{\eps - \expop \eps}} \geq t} 
&~\leq~ 2\exp\inparen{- \frac{c \cdot t^2}{\norm{A_i}_2^2 \cdot \max_i \norm{\eps_i - \expop \eps_i}_{\psi_2}^2}} \\
&~\leq~ 2\exp\inparen{- c' \cdot t^2 \cdot d} \mper
\end{align*}
To bound $\norm{\expop \delta}_{\infty} = \norm{M^{-1} \expop \epsilon}_{\infty}$, we note that by \cref{subgaussian}, $\expop \eps = (\nfrac{2}{d-2}) \cdot \mathbf{1}$, and that conditioned on $E_1$ we actually have $M \succeq (1/3) \cdot I_n + (1/d) \cdot J_n$. Thus,
\[
\norm{\expop \delta}_{\infty} 
~=~ \norm{M^{-1} \expop \epsilon}_{\infty}
~\leq~ \frac{2}{d-2} \cdot \norm{M^{-1} \mathbf{1}}_2
~\leq~ \frac{2}{d-2} \cdot \frac{d}{n} \cdot \norm{\mathbf{1}}_2 ~\leq~ \frac{C_1}{\sqrt{n}} \mper
\]
Choosing $t = C'(\log d)/\sqrt{d}$ in the Hoeffding estimate then proves the claim, since $n = d^2/C$.
%
\end{proof}

\subsection{Spectral bounds on the perturbation}
To prove \cref{thm:main}, we need to show that $Q = I_d + \sum_{i=1}^n \delta_i \cdot X_i \trans{X_i}$ is positive semidefinite.
It suffices to show that the perturbation matrix $P = \sum_{i=1}^n \delta_i \cdot X_i \trans{X_i}$ satisfies $\norm{P} \leq 1$.
To obtain a bound on the spectral norm of the above matrix $P$, we will use the standard argument that it suffices to bound $\norm{Pu}$ for all $u$ in a $\kappa$-net of the sphere $S^{d-1}$ \ie a set of points $N$ such that for all $x \in S^{d-1}$, there exists $y \in N$ with $\norm{x-y}_2 \leq \kappa$. In particular, we will use the following result.
\begin{lemma}[\cite{Ver18}]
\label{net}
Let $A \in \R^{d \times d}$ be a symmetric matrix and let $N$ be a $\kappa$-net of $S^{d-1}$. Then,
\[
\norm{A} ~=~ \sup_{x \in S^{d-1}} \abs{\ip{x}{Ax}} ~\leq~ \frac{1}{1-2\kappa} \cdot \sup_{u \in N}\abs{\ip{u}{Au}}
\]
\end{lemma}
\begin{proof}
Let $x \in S^{d-1}$ be the maximizer of the first quadratic form with $\abs{\ip{x}{Ax}} = \norm{A}$, 
and let $u \in N$ be the closest point. Then
\[ 
\norm{A} - \abs{\ip{u}{Au}} ~\leq~ \abs{\ip{(x-u)}{Ax}} + \abs{\ip{x}{A(x-u)}} ~\leq~ 2 \cdot \norm{Ax}\cdot \norm{x-u} ~\leq~ 2\kappa \cdot \norm{A} \mper \qedhere
\]
\end{proof}
Combined with known estimates on sizes of $\kappa$-nets, stating for example that there exists a $1/4$-net size at most $9^d$ (see \eg ~\cite{Ver18}, Corollary 4.2.13), it suffices to prove that for a \emph{fixed} $u$ in a $1/4$-net $N$, we have $\abs{\ip{u}{Pu}} \leq 1/2$ with probability at least (say) $1 - \exp(-4 d)$ (to allow for a union bound).

We start by considering $\ip{u}{Pu} = \sum_{i=1}^n \delta_i \cdot \ip{u}{X_i}^2$. We will consider the vector $\beta$ with $\beta_i = \ip{u}{X_i}^2$, and split it in two parts according to a threshold $t_0$ to be chosen later.
\[
\beta_{heavy} = \sum_{i \in [n]} \beta_i \cdot e_i \cdot \indicator{\beta_i > t_0} 
\qquad \text{and} \qquad
\beta_{light} = \beta - \beta_{heavy} \mper
\]
We first prove that $\supp(\beta_{heavy}) = \inbraces{i | \beta_i > t_0}$ is bounded in size.
\begin{claim}
\label{support}
There exist $C_1, C_2 > 0$ such that over the choice of $X_1, \ldots, X_n$, we have for $t_0 \geq C_1/d$,
\[
\Pr{X_1, \ldots, X_n}{\abs{\supp(\beta_{heavy})} \geq C_2/t_0} ~\leq~ \exp(-100 \cdot d) \mper
\]
\end{claim}
\begin{proof}
Since $X_i \sim Unif(S^{d-1})$, $\norm{X_i}_{\psi_2} \leq C/\sqrt{d}$, which implies $\Pr{\ip{u}{X_i}^2 \geq t} \leq \exp(-c \cdot t \cdot d)$. By independence, the probability of having $k$ heavy coordinates is at most $n^k \cdot \exp(-c \cdot t \cdot d \cdot k)$, which can be made smaller than $\exp(-100 \cdot d)$ for $k = O(1/t_0)$.
\end{proof}
%
%
We next prove that $\|\beta_{light}\|_2$ is bounded, which allows for concentration bounds on $\ip{\delta}{\beta_{light}}$.
\begin{lemma}[$\ell_2$ bound on $\beta_{light}$]
\label{light}
There exist constants $c_0, C_0 > 0$ such that for $t_0 < c_0$, we have
\[
\Pr{X_1,\ldots,X_n}{\norm{\beta_{light}}^2_2 \geq C_0 \cdot \frac{n}{d^2}} ~\leq~ \exp(-100 \cdot d) \mper
\]
\end{lemma}
\begin{proof}
We will use again that sub-gaussian behavior of $X_i$ implies $\Pr{\beta_i \geq t} \leq \exp(- c \cdot td)$. For $\lambda \in (0,cd/(2t_0))$, we have
\begin{align*}
\Ex{\exp\inparen{\lambda \cdot \min\{\beta_i^2, t_0^2\}}}
&~=~\int_{0}^{exp(\lambda t_0^2)} \mathbb{P}\left(\exp(\lambda \beta_i^2) \geq s \right) \diff s \\
&~\leq~ 1 + \int_{1}^{exp(\lambda t_0^2)} \mathbb{P}\left(\exp(\lambda \beta_i^2) \geq s \right) \diff s \\ 
&~=~ 1 + \int_{0}^{t_0} \mathbb{P}\left(\beta_i^2 \geq t \right) 2\lambda t \exp(\lambda t^2)  \diff t  \\  
&~\leq~ 1 + \int_{0}^{t_0} \exp(\lambda t^2 -cdt)  \lambda t \diff t \\
&~\leq~ 1 + \int_{0}^{t_0} \exp(-cdt/2)  \lambda t \diff t 
~=~ 1 + \frac{8\lambda}{c^2d^2}
\end{align*}
where the last inequality used $\lambda < cd/(2t_0)$. We thus have
\[
\mathbb{E} [\exp(\lambda \cdot \|\beta_{light}\|^2_2)]  ~\leq~ \prod_{i = 1}^n \Ex{\exp\inparen{\lambda \cdot \min\{\beta_i^2, t_0^2\}}}  ~\leq~ (1 + 8\lambda / c^2d^2)^n ~\leq~ \exp(8\lambda n/c^2 d^2) \mper
\]
Taking $c_0 < c/2$ and choosing $\lambda = d$ gives by Markov's inequality,
\[
\Pr{\norm{\beta_{light}}^2_2 \geq C_0 \cdot \frac{n}{d^2}} 
~\leq~ \exp\inparen{\frac{8d \cdot n}{c^2d^2} - \frac{d \cdot C_0 \cdot n}{d^2}} 
~\leq~ \exp(-100 \cdot d) \mcom
\]
since the probability can be made small by an appropriate choice of $C_0$.
\end{proof}
We now have all the ingredients to prove \cref{thm:main}.
%
%
\begin{proofof}{\cref{thm:main}}
Let $E = E_1 \wedge E_2 \wedge E_3$ and note that $\mathbb{P}(E) = 1 - o_d(1)$. Conditioned on $E$, we will show that 
for any $u \in N$, $|\ip{u}{Pu}| = |\ip{\delta}{\beta}| \leq 1/2$ with exponentially high probability, and then take a union bound over N. 
From \cref{bound2} we know that $\norm{\delta}_{\infty} = O(\log (d)/\sqrt{d})$ and \cref{support} implies that $\supp(\beta_{heavy})$ is bounded by $O(1/t_0)$ with probability $1 - \exp(-100 \cdot d)$. Thus,
\[
|\langle \delta, \beta_{heavy}\rangle| 
~\leq~ \| \delta \|_{\infty} \cdot \|\beta_{heavy}\|_1 
~\leq~ O\inparen{\log(d) / (\sqrt{d} \cdot t_0)} ~\leq~ \frac16 \quad \text{for}~ t_0 = O(d^{-1/4})
\]
Now we turn to the contribution from $\beta_{light}$, which can be written as
\[
|\langle \delta, \beta_{light}\rangle| 
~=~ |\langle M^{-1}\epsilon, \beta_{light}\rangle| 
~=~  |\langle M^{-1}\epsilon, \beta_{light}\rangle| 
~=~  |\langle \epsilon, M^{-1}\beta_{light}\rangle| 
\]
By \cref{operator} and \cref{light}, we have $\|M^{-1}\beta_{light}\|_2^2 \leq \|M^{-1}\|^2_2\|\|\beta_{light}\|_2^2 = O(n/d^2)$. 
Since $\epsilon$ and $M^{-1}\beta_{light}$ are independent, and $\epsilon$ conditioned on $E_2$ has sub-gaussian norm $O(1/\sqrt{d})$ by \cref{bound1}, Hoeffding's inequality implies that,
\[
\Pr{ \left|\langle \epsilon - \expop \epsilon, M^{-1}\beta_{light}\rangle \right| > \frac{1}{6} }
~\leq~ 2 \cdot \exp \left(-\frac{c/36}{\|M^{-1}\beta_{light}\|_2^2 \cdot \|\epsilon - \expop \epsilon\|^2_{\psi_2}}\right) 
~=~ 2 \cdot \exp (-c' \cdot d^3/n) \mcom
\]
which is at most $\exp(-100 \cdot d)$ when $n < d^2/C$ for an appropriate $C$.
Finally, note that using \cref{subgaussian} and \cref{light}, with probability $1 - \exp(-100 \cdot d)$,
\[
\abs{\ip{\expop \eps}{M^{-1} \beta_{light}}} ~=~ \frac{2}{d-2} \cdot \abs{\ip{\mathbf{1}}{M^{-1} \beta_{light}}} ~\leq~ \frac{2}{d-2} \cdot \sqrt{n} \cdot \norm{M^{-1} \beta_{light}} ~=~ O\inparen{\frac{1}{d} \cdot \sqrt{n} \cdot \frac{\sqrt{n}}{d}} \mper
\]
%
%
%
Thus, we have that conditioned on $E$, $\Pr{\abs{\ip{\delta}{\beta}} \geq  1/2} \leq \exp\inparen{-10 \cdot d}$ when $n < d^2/C$ for a large enough $C$. We can now take the union bound on the net to show
\begin{align*}
\mathbb{P}\left(\sup_{u\in N} \left |\sum_{i=1}^n \delta_i \left \langle u,X_i \right \rangle^2 \right | \geq \frac{1}{2} \right) 
&~\leq~ 
\mathbb{P}\left( \sup_{u\in N} \left |\sum_{i=1}^n \delta_i \left \langle u,X_i \right \rangle^2 \right | \geq \frac{1}{2} \Bigg | E\right)\cdot \mathbb{P}(E) + \mathbb{P}(E^c) \\ 
&~\leq~ 9^d \cdot \Bigg[\mathbb{P} \Big(|\langle \delta, \beta \rangle |\geq 1/2 \Big| E \Big) \cdot \mathbb{P}(E) \Bigg ]  + \mathbb{P} (E^c) \\
&~\leq~ 9^d \cdot \exp(- 10 \cdot d) + o_d(1) ~=~ o_d(1) \mper \qedhere
\end{align*}
Combined with \cref{net}, this implies $\norm{P} \leq 1$ and hence $Q \succeq 0$ with probability $1 - o_d(1)$.
\end{proofof}

\section*{Acknowledgements}
We are grateful to Aaron Potechin for introducing us to this problem, to Goutham Rajendran and Jeff Xu for enlightening conversations, and to Afonso Bandeira for encouraging us to publish this note.

\bibliographystyle{alphaurl}
\begingroup

\raggedright 
\bibliography{macros,madhur}

\endgroup
\
\end{document}